\documentclass[12pt, reqno]{amsart}
\usepackage{amsmath, amsthm, amscd, amsfonts, amssymb, graphicx, color}
\usepackage[bookmarksnumbered, colorlinks, plainpages]{hyperref}
\hypersetup{colorlinks=true,linkcolor=red, anchorcolor=green, citecolor=cyan, urlcolor=red, filecolor=magenta, pdftoolbar=true}
\input{mathrsfs.sty}
\usepackage{xcolor}
\usepackage{tikz}
\usetikzlibrary{decorations.pathreplacing,angles,quotes}
\definecolor{cof}{RGB}{219,144,71}
\definecolor{pur}{RGB}{186,146,162}
\definecolor{greeo}{RGB}{91,173,69}
\definecolor{greet}{RGB}{52,111,72}
\usepackage{xcolor}
\usepackage{ulem}
\newtheorem{theorem}{Theorem}[section]
\newtheorem{lemma}[theorem]{Lemma}
\newtheorem{proposition}[theorem]{Proposition}
\newtheorem{cor}[theorem]{Corollary}
\theoremstyle{definition}

\newtheorem{example}[theorem]{Example}

\theoremstyle{remark}
\newtheorem{remark}[theorem]{\bf{Remark}}
\numberwithin{equation}{section}
\newcommand{\sm}{{\rm Sm}\,}
\newcommand{\ksm}{{\rm \text{$k$-}Sm}\,}
\newcommand{\rint}{{\rm Int_r}\,}
\newcommand{\ext}{{\rm Ext}\,}

\newcommand{\aff}{\rm aff}
\newcommand{\co}{\rm co}
\newcommand{\spn}{{\rm span}}
\allowdisplaybreaks
\begin{document}
	
	\title[Preservation of parallel pairs ]{A study on  preservation of parallel pairs and triangle equality attainment pairs }
	\author[Manna, Mandal, Paul and Sain  ]{Jayanta Manna, Kalidas Mandal,    Kallol Paul and Debmalya Sain }

	\address[Manna]{Department of Mathematics\\ Jadavpur University\\ Kolkata 700032\\ West Bengal\\ INDIA}
	\email{iamjayantamanna1@gmail.com}
	
	\address[Mandal]{Department of Mathematics\\ Jadavpur University\\ Kolkata 700032\\ West Bengal\\ INDIA}
\email{kalidas.mandal14@gmail.com}
	
	\address[Paul]{Vice-Chancellor, Kalyani University \& Professor of Mathematics\\ Jadavpur University (on lien) \\ Kolkata \\ West Bengal\\ INDIA}
	\email{kalloldada@gmail.com}
	
		\newcommand{\acr}{\newline\indent}
	\address[Sain]{Department of Mathematics\\ Indian Institute of Information Technology, Raichur\\ Karnataka 584135 \\INDIA}
	\email{saindebmalya@gmail.com}

	
	\begin{abstract}
		This paper investigates the structure and preservation of parallel pairs and triangle equality attainment (TEA) pairs in normed linear spaces. We begin by providing functional characterizations of these pairs in  normed linear spaces and provide a characterization of finite-dimensional Banach spaces with numerical index one.  We then study parallel pairs and TEA pairs in the $p$-direct sum of normed linear spaces. Next, we show that in finite-dimensional Banach spaces, a bijective bounded linear operator preserves TEA pairs if and only if it preserves parallel pairs, and this equivalence remains valid in finite-dimensional polyhedral spaces for operators with rank greater than one. We further explore some geometric consequences related to this preservation and provide a characterization of isometries on certain polyhedral Banach spaces as norm-one bijective operators that preserve the norm at extreme points and also preserve the parallel or TEA pairs.

	\end{abstract}
	
	\maketitle
	
	\section{Introduction.}
    The notions of parallel pairs and triangle equality attainment (TEA) pairs are crucial in understanding the geometric structure of normed linear spaces and the nature of the vectors within the spaces. These pairs have significant implications in the study of the geometric properties of the spaces such as strict convexity, smoothness. They also provide insights into extremal properties of the unit balls of normed linear spaces. Moreover, their study has a close connection with many other fields of mathematics, including approximation theory, signal processing, data science and numerical analysis. For more on the study of parallel pairs and  TEA pairs, the readers are referred to \cite{BB01,BCMWZ19,KLPS25,KLPS24,LTWW25,MSP19, NT02, W17,ZM16,ZM15}. The purpose of this article is to characterize parallel pairs and  TEA pairs in  normed linear spaces and to study the preservation of these pairs by a linear operator defined on normed linear spaces.  Let us now introduce the relevant terminology and notation used in this article.
    
	The letters  $ \mathbb{X}, \mathbb{Y} $ denote normed linear spaces. Throughout this article, we consider \textit{real} normed linear spaces of dimension greater than one. Let $ \mathbb{L}(\mathbb{X}, \mathbb{Y})$ be  the space of all bounded linear operators from $ \mathbb{X} $ to $ \mathbb{Y},$  when $ \mathbb{X} = \mathbb{Y} $, we simply write $ \mathbb{L}(\mathbb{X}).$ The dual space of $ \mathbb{X} $ is denoted by $\mathbb{X}^*. $ The unit ball and unit sphere of $ \mathbb{X}$ are given by  
	$B_{\mathbb{X}} = \{x\in \mathbb{X} : \|x\|\leq 1\}\text{ and } S_{\mathbb{X}} = \{x\in \mathbb{X} : \|x\|=1\}, $
	respectively.   Given a nonempty set $ A \subset \mathbb{X} $, its convex hull is denoted by $ \co(A).$ For a nonempty convex set $ A \subset \mathbb{X} $, a point $ x \in A $ is an extreme point if $ x = (1-t)y + t z $ for  $ t \in (0,1) $ and $ y, z \in A $ imply $ x = y = z $. The set of all extreme points of $ A $ is denoted by $ \ext~ A $. A finite-dimensional Banach space $ \mathbb{X}$ is polyhedral if $ \ext~ B_{\mathbb{X}} $ is finite.  A convex subset $ F \subset S_{\mathbb{X}} $ is a face of $ B_{\mathbb{X}} $ if $ (1-t)x_1 + t x_2 \in F $ for $ x_1, x_2 \in S_{\mathbb{X}} $ and $ 0 < t < 1 $ imply $ x_1, x_2 \in F $. A maximal face of $ B_{\mathbb{X}} $ is called a facet.  	The relative interior of  a nonempty set $U\subset \mathbb{X}$ is denoted by $\text{Int}_r~U$ and is defined by  $\text{Int}_r~U=\{x\in U : \text{ there exists }\epsilon>0\text{ such that }  \text{aff} (U)\cap B(x,\epsilon)\subset U\},$ where $\text{aff}(U)$ is the intersection of all affine spaces containing $U.$ Note that an affine space is a translation of a subspace of a vector space.
	
	For any non-zero element $x \in \mathbb{X}$, we define the set of all supporting functionals at $x$ as:
	$
		J(x) = \{ f \in S_{\mathbb{X}^*} : f(x) = \| x \| \}.$
	A non-zero element $x \in \mathbb{X}$ is termed as a smooth point if there exists a unique supporting functional at $x$, i.e., $J(x)$ is singleton. Furthermore, $x \in \mathbb{X}$ is said to be a {$k$-smooth point} (or to have smoothness of order $k$) if $\dim(\text{span } J(x)) = k$, where $k \in \mathbb{N}$. For more on smoothness and $k$-smoothness, readers may consult \cite{DMP22,KS05,LR07, MDP22, MP20,PSG16,SPMR20,W18} and references therein. We denote the collection of all smooth points in $\mathbb{X}$ by $\sm\mathbb{X}$ and the collection of all $k$-smooth points in $\mathbb{X}$ by $\ksm\mathbb{X}.$ 
\smallskip

    We next mention the definition of parallel pair and   triangle equality attaining (TEA) pair. For $x,y \in \mathbb X,$
	\begin{itemize}
		\item[(i)]	$(x,y)$ is a parallel pair in $\mathbb{X}$ if
		$\|x+\lambda y\|=\|x\|+\|y\|$ for some scalar $\lambda$   with $|\lambda|=1,$
		\item[(ii)]$(x,y)$ is a triangle equality attaining (TEA) pair  of $\mathbb{X}$ if
		$\|x+y\|=\|x\|+\|y\|.$
	\end{itemize}
	In the framework of inner product spaces, the norm-parallel relation is exactly the usual vectorial parallel relation, i.e., $x || y$ if and only if $x$ and $y$ are linearly dependent. In the setting of normed linear spaces, two linearly dependent vectors are norm-parallel, but the converse may not hold in general. To see this, consider the vectors $x=(1, 0)$ and $y=(1, 1)$ in the space $\ell_{\infty}^2.$ Then $(x,y)$ is a parallel pair and they are linearly independent.  Note that the norm-parallelism is symmetric and
	$\mathbb R$-homogeneous. Also, it can be easily verified that this relation is not transitive by considering  the vectors $x=(1, 0), y=(1, 1)$ and $z=(0,1)$ in the space $\ell_{\infty}^2.$ \\
For  $T\in\mathbb{L}(\mathbb{X},\mathbb{Y}), $
	\begin{itemize}
		\item[(i)]  $T$ is said to preserve parallel pairs if for any $x,y\in \mathbb{X},$ $(x, y)$ is a parallel pair in $\mathbb {X} \implies (Tx, Ty)$ is a parallel pair in $\mathbb {Y}$ and 
		\item[(ii)]$T$ is said to preserve TEA pairs if for any $x,y\in \mathbb{X},$ $(x, y)$ is a TEA pair in $\mathbb {X} \implies (Tx, Ty)$ is a TEA pair in $\mathbb {Y}$.
	\end{itemize}
    Note that a normed linear space is strictly convex if and only if for any $x,y\in \mathbb{X},~(x,y)$ is a parallel pair $\iff x,y$ are linearly dependent \cite{MSP19}. So every bounded linear operator in a strictly convex normed linear space preserves both parallel and TEA pairs.
    
  The central aim of this paper is to characterize parallel pairs and TEA pairs in normed linear spaces and to study bounded linear operators that preserve these pairs. First, we provide necessary and sufficient conditions for both parallel and TEA pairs in  normed linear spaces. Next, we present a characterization of finite-dimensional Banach spaces with numerical index one. We then examine TEA pairs and parallel pairs in the $p$-direct sum of two normed linear spaces. After that we turn our attention to the preservation of parallel pairs and TEA pairs by a linear operator defined on  normed linear spaces. We first show that if a bounded linear operator preserves TEA pairs of any face of the unit ball of a normed linear space, then the image of any relative interior point of that face under the operator can not be zero unless the image of the entire face is zero. We then establish that for finite-dimensional Banach spaces the preservation of TEA pairs by a bijective bounded linear operator is equivalent to the preservation of parallel pairs. Furthermore, we show that in the case of finite-dimensional polyhedral spaces, this equivalence holds for any operator of rank greater than one. Additionally, we explore various geometric aspects of the relevant spaces, particularly in connection with the preservation of parallel and TEA pairs. Finally, we characterize isometries on a class of polyhedral Banach spaces as norm-one bijective operators that preserve the norm at each extreme point of the unit ball of the domain and also preserve parallel (or TEA) pairs.
	\section{Main Results}
	We begin this section with the following characterizations of parallel pairs and TEA pairs in a normed linear space.
	\begin{proposition}\label{par char}
		Let $\mathbb{X}$ be a normed linear space and let  $x,y\in \mathbb{X}$ be non-zero. Then
		\begin{itemize}
			\item[(i)]  $(x,y)$ is a parallel pair  if and only if there exist  $f\in J(x)$ and $g\in J(y)$ such that $g=\lambda f$ for some scalar $\lambda$  with  $|\lambda|=1,$
			\item[(ii)] $(x,y)$ is  a TEA pair if and only if  $J(x)\cap J(y)\neq\emptyset.$
		\end{itemize}
	\end{proposition}
	\begin{proof}
		(i) Let $f\in J(x)$ and $\lambda f \in J(y).$ Then $f(x)=\|x\|$ and $\lambda f(y)=\|y\|.$ Thus,
		\[\|x\|+\|y\|=f(x)+\lambda f(y)=f(x+\lambda y)\leq\|x+\lambda y\|\leq \|x\|+\|y\|.\]
		This implies that $\|x+\lambda y\|= \|x\|+\|y\|,$ i.e., $(x, y)$ is a parallel pair.\\
		Conversely, let  $(x,y)$ be a parallel pair. So $\|x+\lambda y\|= \|x\|+\|y\|$  for some scalar $\lambda$  with  $|\lambda|=1.$ Let $f\in J(x+\lambda y).$ This implies that \[f(x)+\lambda f(y)=f(x+\lambda y)=\|x+\lambda y\|= \|x\|+\|y\|.\]
		Since $|f(x)|\leq\|x\|$ and $|\lambda f(y)|\leq\|y\|,$ it follows that $f(x)=\|x\|$ and $\lambda f(y)=\|y\|.$ Therefore, $f\in J(x)$ and $\lambda f\in J(y)$ and this completes the proof.\\
		(ii)  Taking $\lambda =1$ in (i), we get the required characterization.
	\end{proof}
	Next, we present the following observation concerning the parallel pairs and TEA pairs that includes an element which is both smooth and exposed. Note that an element $x\in S_{\mathbb{X}}$ is called an exposed point of $B_{\mathbb{X}}$ if there exists $f \in J(x)$ such that $f$ attains its norm uniquely at $\pm x.$
	\begin{proposition}
		Let $\mathbb{X}$ be a normed linear space and $x\in S_{\mathbb{X}}.$ If $x$ is both smooth and exposed point of $B_{\mathbb{X}}$ then for any $y\in S_{\mathbb{X}}$ 
		\begin{itemize}
			\item [(i)] $(x,y)$ is a parallel pair if and only if $x=\pm y,$
			\item [(ii)] $(x,y)$ is a TEA pair if and only if $x=y.$
		\end{itemize}
	\end{proposition}
	\begin{proof}
		Let $x,y\in S_{\mathbb{X}}$ and let $x$ be a smooth and exposed point of $B_{\mathbb{X}}.$  Then there exists a unique $f \in J(x)$ such that $f$ attains its norm uniquely at $\pm x$.\\
		(i) If $x=\pm y$ then clearly $(x,y)$ is a parallel pair. Conversely, let $(x,y)$ be a parallel pair. Then from  Proposition \ref{par char}, it follows that there exists $g\in J(y)$ such that $g $ and $ f$ are linearly dependent. Then $|f(y)|=|g(y)|=\|y\|.$ Thus $f$ attains its norm at $y$ and so $x=\pm y.$ \\
		(ii) If $x=y$ then clearly $(x,y)$ is a TEA pair. Conversely, let $(x,y)$ be a TEA pair. Then from Proposition \ref{par char}, it follows that  $f\in J(y).$ Thus $x=y.$ 
	\end{proof}
	
	The above result fails to hold if either $x$ is smooth but not exposed, or $x$ is exposed but not smooth. The following  example illustrates the scenario.
	\begin{example} Let $\mathbb{X}=\ell_{\infty}^2$ and let  $x_1=(1,0),~y_1=(1,\frac{1}{2}).$  Here $x_1,y_1\in S_{\mathbb{X}}$ and are smooth. Clearly, $(x_1,y_1)$ is a TEA pair as well as parallel pair but $x_1,y_1$ are linearly independent. Again, let  $x_2=(1,1),~y_2=(1,-1).$  Here $x_2,y_2\in S_{\mathbb{X}}$ and are exposed point of $B_{\mathbb{X}}.$ Clearly, $(x_2,y_2)$ is a TEA pair as well as parallel pair but $x_2,y_2$ are linearly independent.   
	\end{example}
	
	Next, we give a characterization for finite-dimensional Banach spaces with numerical index $1.$ First, we recall that for a Banach space $\mathbb{X},$  the numerical index $ n(X) $ of $\mathbb{X}$ is defined as  
	\[
	n(\mathbb{X}) = \inf \{ v(T) : T \in \mathbb{L}(\mathbb{X}), \, \|T\| = 1 \},
	\]
	where $v(T)$ is the numerical radius  of $ T,$  defined by 
	\[
	v(T) = \sup \{ |f(Tx)| : x \in S_{\mathbb{X}}, \, f \in S_{\mathbb{X}^*}, \, f(x) = 1 \}.
	\]

	\begin{theorem}
		Let $\mathbb{X}$ be a finite-dimensional  Banach space. Then $n({\mathbb{X}})=1$ if and only if  $(x,y)$  is a parallel pair for any $x\in \ext~ B_{\mathbb{X}}$ and $y\in S_{\mathbb{X}}.$
	\end{theorem}
	\begin{proof}
		Let $n({\mathbb{X}})=1.$ Suppose $x\in \ext~ B_{\mathbb{X}}$ and $y\in S_{\mathbb{X}}.$ Now, it follows from \cite[Th. 3.1]{M71} that $|f(z)|=1$ for every $z\in \ext~B_{\mathbb{X}}$  and $f\in \ext~B_{\mathbb{X}^*}.$ Let $f\in \ext~J(y)$ then as $x\in \ext  ~B_{\mathbb{X}},$ $|f(x)|=1.$ Thus $\lambda f \in J(x)$ for some scalar $\lambda$ with $|\lambda|=1.$ Therefore, $(x,y)$ is a parallel pair.
		
		Conversely, suppose for any $x\in \ext~ B_{\mathbb{X}}$ and $y\in S_{\mathbb{X}},$  $(x,y)$  is a parallel pair. Let $T \in \mathbb{L}(\mathbb{X})$ such that $ \|T\| = 1.$ Since $\mathbb {X}$ is  finite-dimensional,  there exists $w \in \ext~B_{\mathbb{X}}$ such that $||Tw||=1.$ Now, it follows from the hypothesis that $(w,Tw)$ is a parallel pair. From  Proposition \ref{par char}, it follows that there exists $f\in S_{\mathbb{X}^*}$ such that $f\in J(w)$ and $\lambda f\in J(Tw),$ where $|\lambda|=1.$ Then
		\[
		1=\lambda f(Tw)=|f(Tw)|\leq \sup \{ |g(Tz)| : z \in S_{\mathbb{X}}, \, f \in S_{\mathbb{X}^*}, \, g(z) = 1 \}\leq  1.
		\]
		Thus, for any norm one operator $T \in \mathbb{L}(\mathbb{X}),$ $v(T)=1.$  Therefore,
		\[
		n(\mathbb{X}) = \inf \{ v(T) : T \in \mathbb{L}(\mathbb{X}), \, \|T\| = 1 \}=1.
		\]
		
	\end{proof}
Given  normed linear spaces $\mathbb{X}$ and $\mathbb{Y},$ we study parallel pairs and TEA pairs  in the $p$-direct sum space $\mathbb{Z}=\mathbb{X}\bigoplus_p\mathbb{Y},$ where $1\leq p\leq \infty.$  Our goal is to establish connections between the parallel pairs  and  the TEA pairs in the individual spaces $\mathbb{X}, \mathbb{Y}$ and those in $\mathbb{Z}$. We begin by considering the case when $1 \leq p < \infty$.
	\begin{theorem}\label{par 1leq p<infty}
	Let $\mathbb{X}$ and $\mathbb{Y}$ be normed linear spaces and $\mathbb{Z}=\mathbb{X}\bigoplus_p\mathbb{Y},$ with $1\leq p< \infty.$ Let $z_1, z_2\in \mathbb{Z}$ such that $z_1=(x_1,y_1)$ and $z_2=(x_2,y_2),$ where $x_1,x_2\in \mathbb{X}\setminus \{0\}$ and $y_1,y_2\in \mathbb{Y}\setminus \{0\}.$
	\begin{itemize}
		\item[(i)] If $(z_1, z_2)$  is a parallel pair in $\mathbb{Z}$ then $(x_1, x_2)$  and $(y_1, y_2)$ are parallel pairs in $\mathbb{X}$ and $\mathbb{Y},$ respectively.
        \item[(i)] If $(z_1, z_2)$ is a TEA pair in $\mathbb{Z}$ then $(x_1, x_2)$  and $(y_1, y_2)$ are TEA pairs in $\mathbb{X}$ and $\mathbb{Y},$ respectively.
	\end{itemize}
	 
\end{theorem} 
\begin{proof}
	(i) Suppose $(z_1, z_2)$  is a parallel pair in $\mathbb Z.$ Then by Proposition \ref{par char} there exist $f\in J(z_1)$ and $g\in J(z_2)$ such that $f=\lambda g$ for some $\lambda$ with $|\lambda|=1.$ Then by \cite[Prop. 5.1]{CKS23}, it follows that $f=(k_1\phi_1,l_1\psi_1)$ and $g=(k_2\phi_2,l_2\psi_2),$ where for $i=1,2,$
    $\phi_i\in J(x_i)$ and  $\psi_i\in J(y_i)$  and  
    \[k_i={ \begin{cases}
        1,&\text{ if }p=1\\
        \frac{\|x_i\|^{p-1}}{(\|x_i\|^p+ \|y_i\|^p)^{(1-\frac1p)}},&\text{ if }1<p<\infty
    \end{cases}}\text{ and }l_i={ \begin{cases}
        1,&\text{ if }p=1\\
        \frac{\|y_i\|^{p-1}}{(\|x_i\|^p+ \|y_i\|^p)^{(1-\frac1p)}},&\text{ if }1<p<\infty
     \end{cases}}.\]
     Then $f=\lambda g\implies (k_1\phi_1,l_1\psi_1)=\lambda (k_2\phi_2,l_2\psi_2).$ This shows that $\phi_1=\frac{\lambda k_2}{k_1}\phi_2$ and $\psi_1=\frac{\lambda l_2}{l_1}\psi_2.$ Since $\|\phi_i\|=\|\psi_i\|=1$ for $i=1,2,$ it follows that $\left|\frac{\lambda k_2}{k_1}\right|=\left|\frac{\lambda l_2}{l_1}\right|=1.$ Then from Proposition \ref{par char}, it follows that $(x_1, x_2)$  and $(y_1, y_2)$ are parallel pairs in $\mathbb{X}$ and $\mathbb{Y},$ respectively.
  
  (ii) Using the similar arguments as in (i) for $\lambda=1,$ we obtain (ii).
	
\end{proof} 
The converse of the above result does not necessarily hold, as demonstrated by the following example.
\begin{example}
		Let $\mathbb{Z}=\ell_1^2\bigoplus_p\ell_1^2$ and $1<p<\infty.$ Consider $x_1,x_2,y_1,y_2\in \ell_1^2$ such that $x_1=(1,0), x_2=(2,0), y_1=(3,0)$ and $y_2=(1,0)$. It is easy to verify that $(x_1, x_2)$ and $(y_1, y_2)$ are TEA pairs in $\ell_1^2.$ Now let $z_1=(x_1,y_1)$ and $z_2=(x_2,y_2).$ Then $z_1,z_2\in \mathbb{Z}$ and $\|z_1\|+\|z_2\|=(1+3^p)^{\frac1p}+(1+2^p)^{\frac1p},$ but
      \[ \|z_1+\lambda z_2\|=\left(\left|1+ 2\lambda\right|^p+\left|3+\lambda\right|^p\right)^{\frac1p}
           \neq (1+3^p)^{\frac1p}+(1+2^p)^{\frac1p}, \]
             for any $ \lambda$ with $|\lambda|=1.$
        This shows that $(z_1,z_2)$ is not a parallel pair in $\mathbb {Z}.$
	\end{example}
The following result shows that the converse of Theorem \ref{par 1leq p<infty} (ii) holds for $p=1.$
    \begin{theorem}\label{par 1}
	Let $\mathbb{X}$ and $\mathbb{Y}$ be normed linear spaces and $\mathbb{Z}=\mathbb{X}\bigoplus_1\mathbb{Y}.$ Let $z_1, z_2\in \mathbb{Z}$ be such that $z_1=(x_1,y_1)$ and $z_2=(x_2,y_2),$ where $x_1,x_2\in \mathbb{X}\setminus \{0\}$ and $y_1,y_2\in \mathbb{Y}\setminus \{0\}.$ Then
	  $(z_1, z_2)$  is a TEA pair in $\mathbb{Z}$ if and only if $(x_1, x_2)$  and $(y_1, y_2)$ are TEA pairs in $\mathbb{X}$ and $\mathbb{Y},$ respectively.
    \end{theorem}
    \begin{proof}
	The necessary part follows from Theorem \ref{par 1leq p<infty}. For the sufficient part, suppose $(x_1, x_2)$  and $(y_1, y_2)$ are TEA pairs in $\mathbb{X}$ and $\mathbb{Y},$ respectively. Then by Proposition \ref{par char}, we have $J(x_1)\cap J(x_2)\neq \emptyset$  and $ J(y_1)\cap J(y_2)\neq \emptyset.$ Let $f\in J(x_1)\cap J(x_2)$  and $g\in J(y_1)\cap J(y_2).$ Let $\phi=(f,g).$ Then again by \cite[Prop. 5.1]{CKS23}, it follows that $\phi\in J(z_1)$ and $\phi\in J(z_2).$ This shows that $J(z_1)\cap J(z_2)\neq \emptyset.$ This completes the proof.
\end{proof} 
In the next example, we show that even for $p=1$ the converse of Theorem \ref{par 1leq p<infty} (i) fails to hold.
     \begin{example}
		Let $\mathbb{Z}=\ell_1^2\bigoplus_1\ell_1^2$ and let $x_1=(1,0), x_2=\left(\frac{1}{2},\frac{1}{2}\right), y_1=(0,1)$ and $y_2=\left(\frac{1}{2},-\frac{1}{2}\right)$. It is easy to verify that $(x_1, x_2)$ and $(y_1, y_2)$ are parallel pairs in $\ell_1^2.$ Now let $z_1=(x_1,y_1)$ and $z_2=(x_2,y_2).$ Then for each $\lambda$ with $|\lambda|=1,$ we have 
        \begin{eqnarray*}
          \|z_1+\lambda z_2\|&=&\left\|\left(1+\frac{\lambda}{2}, \frac{\lambda}{2}\right)\right\|+\left\|\bigg(\frac{\lambda}{2},1-\frac{\lambda}{2}\bigg)\right\|\\
           &=&\left|1+\frac{\lambda}{2}\right|+ \left|\frac{\lambda}{2}\right|+ \left|\frac{\lambda}{2}\right|+\left|1-\frac{\lambda}{2}\right|\\
          &=& 3,
        \end{eqnarray*}
        and $\|z_1\|=\|z_2\|=2.$ So $(z_1,z_2)$ is not a parallel pair in $\mathbb {Z}.$
	\end{example}
   Next, we examine the parallel pairs and the TEA pairs in the space $\mathbb{Z}=\mathbb{X}\bigoplus_p\mathbb{Y},$ where $1 \leq p < \infty,$  considering only those elements that have at least one coordinate zero.
\begin{theorem}
	Let $\mathbb{X}$ and $\mathbb{Y}$ be normed linear spaces and $\mathbb{Z}=\mathbb{X}\bigoplus_p\mathbb{Y},$ with $1\leq  p< \infty.$ Let $z_1, z_2\in \mathbb{Z}\setminus \{0\}$ such that $z_1=(x_1,y_1)$ and $z_2=(x_2,y_2).$
	 \begin{itemize}
     
	 	\item[(i)] If $x_1=x_2=0$ then  $(z_1,z_2)$ is a parallel (or TEA) pair in $\mathbb{Z}$ if and only if $(y_1,y_2)$ is a parallel (or TEA) pair in $\mathbb{Y}.$  \\
        Analogous result holds for $y_1=y_2=0.$ 
	 	\item[(ii)] If $x_1=0,~y_2=0,$ or $x_2=0,~y_1=0$ then 
        \begin{itemize}
            \item[(a)] $(z_1,z_2)$ is  a TEA pair in $\mathbb{Z}$ if $p=1,$
	 	\item[(b)] $(z_1,z_2)$ is not a parallel pair $\mathbb{Z}$ if $1<  p< \infty.$
        \end{itemize}
	 \end{itemize}
\end{theorem}
\begin{proof}
(i)Let $x_1=x_2=0.$ Since $z_1, z_2\in \mathbb{Z}\setminus \{0\},$ it follows that $y_1,y_2\in \mathbb{Y}\setminus \{0\}.$ For the sufficient part suppose that  $(y_1, y_2)$ is a parallel pair in $\mathbb{Y}.$  Then by Proposition \ref{par char}, we have $g_1\in J(y_1)$ and $ g_2\in J(y_2)$ such that $g_2=\lambda g_1$ for some $|\lambda |=1.$  Let $\phi_1=(0,g_1)$ and $\phi_2=(0,g_2).$   From  \cite[Prop. 5.1]{CKS23}, it follows that $\phi_1\in J(z_1)$ and $\phi_2\in J(z_2).$ Clearly, $\phi_2=\lambda \phi_1$ and so $(z_1,z_2)$ is a parallel pair in $\mathbb{Z}.$\\
By choosing $\lambda =1,$ the result for TEA pair can be derived directly from the previous argument.
  
  (iii) We prove the case where $x_1=0$ and $~y_2=0.$ The other case follows in a similar manner. 
  \begin{itemize}
      \item[(a)] Let $p=1.$ Then
	\begin{eqnarray*}
		\|z_1+ z_2\|&=&\|(x_1,y_1)+(x_2,y_2)\|\\
		&=&\|(x_2, y_1)\|\\
		&=&\|x_2\|+\|y_1\|\\
		&=&\|(x_1,y_1)\|+ \|(x_2,y_2)\|\\
		&=&\|z_1\|+\|z_2\|.
	\end{eqnarray*}
	 \item[(b)] Let $1<  p< \infty.$ If possible, suppose $(z_1,z_2)$ is a parallel pair in $\mathbb {Z}.$  Then it follows from \cite[Prop. 5.1]{CKS23} that $f=(0,\phi_1)\in J(z_1)$ and $g=(\psi_2,0)\in J(z_2),$ where $\phi_1\in J(y_1)$  and $\psi_2\in J(x_2).$ Since $(z_1,z_2)$ is a parallel pair, it follows from \ref{par char} that $f=\lambda g$ for some $\lambda$ with $|\lambda|=1.$   Then $f=\lambda g$ implies that  $\phi_1=0$ and $\psi_2=0,$  which leads us to a contradiction that $f\in J(z_1)$ and $g\in J(z_2).$ So $(z_1,z_2)$ is not a parallel pair in $\mathbb {Z}.$
  \end{itemize}
\end{proof}

In the following result, we provide a complete characterization of the parallel (or TEA) pairs in the space $\mathbb{Z}=\mathbb{X}\bigoplus_{\infty}\mathbb{Y}.$ 
\begin{theorem}\label{par infty}
		Let $\mathbb{X}$ and $\mathbb{Y}$ be normed linear spaces and $\mathbb{Z}=\mathbb{X}\bigoplus_{\infty}\mathbb{Y}.$ Let $z_1, z_2\in \mathbb{Z}\setminus\{0\}$ such that $z_1=(x_1,y_1)$ and $z_2=(x_2,y_2),$ where $x_1,x_2\in \mathbb{X}$ and $y_1,y_2\in \mathbb{Y}.$  Then the following statements hold:
		\begin{itemize}
			\item [(i)] If $\|x_1\|\geq\|y_1\|$ and $\|x_2\|\geq\|y_2\|$ then $(z_1, z_2)$  is a parallel (or TEA) pair in $\mathbb{Z}$ if and only if  $(x_1, x_2)$ is a parallel (or TEA) pair in $\mathbb{X}.$
            \item [(ii)]If $\|x_1\|\leq\|y_1\|$ and $\|x_2\|\leq\|y_2\|$ then $(z_1, z_2)$  is a parallel (or TEA) pair in $\mathbb{Z}$ if and only if $(y_1, y_2)$ is a parallel (or TEA) pair in $\mathbb{Y}.$
     \item [(iii)] If one of the followings hold:
\begin{itemize}
    \item[(a)]$\|x_1\|>\|y_1\|$ and $\|x_2\|<\|y_2\|$
    \item[(b)] $\|x_1\|<\|y_1\|$ and $\|x_2\|>\|y_2\|,$
\end{itemize}
then $(z_1,z_2)$ is not a parallel pair in $\mathbb{Z}.$
		\end{itemize}
	\end{theorem} 
    \begin{proof} 
         (i) First, suppose $(z_1, z_2)$  is a parallel pair in $\mathbb{Z}.$ As $\|x_1\|\geq\|y_1\|$ and $\|x_2\|\geq\|y_2\|,$ it follows from  \cite[Prop. 5.1]{CKS23} that $f=(\phi_1,0)$ and $g=(\phi_2,0),$ where $\phi_i\in J(x_i)$ for $i=1,2.$ Then $f=\lambda g\implies \phi_1=\lambda \phi_2.$ Then by Proposition \ref{par char}, we have $(x_1, x_2)$ is a parallel pair in $\mathbb{X}.$ 
          Using similar argument as above for $\lambda=1,$  we conclude that if $(z_1, z_2)$  is a TEA pair then  $(x_1, x_2)$ is a TEA pair in $\mathbb{X}.$\\
          Now suppose $(x_1, x_2)$ is a parallel pair in $\mathbb{X}.$ Then it follows from Proposition \ref{par char} there exist $f\in J(x)$ and $g\in J(y)$ such that $f=\lambda g.$ Now as $\|x_1\|\geq\|y_1\|$ and $\|x_2\|\geq\|y_2\|,$ again it follows from  \cite[Prop. 5.1]{CKS23} that $\psi_1=(f,0)\in J(z_1)$ and $\psi_2=(g,0)\in J(z_2).$ This shows that $\psi_1=\lambda \psi_2$ and therefore, $(z_1,z_2)$ is a parallel pair in $\mathbb{Z}.$

          (ii) Using similar argument as in (i) we obtain (ii).

          (iii) If possible, suppose $(z_1,z_2)$ is a parallel pair in $\mathbb {Z}$ and $\|x_1\|>\|y_1\|$ and $\|x_2\|<\|y_2\|.$ It follows from  \cite[Prop. 5.1]{CKS23} that $f=(\phi_1,0)\in J(z_1)$ and $g=(0, \psi_2)\in J(z_2),$ where $\phi_1\in J(x_1)$  and $\psi_2\in J(y_2).$ Since $(z_1,z_2)$ is a parallel pair, it follows from \ref{par char} that $f=\lambda g$ for some $\lambda$ with $|\lambda|=1.$   Then $f=\lambda g\implies \phi=0$ and $\psi=0,$  which leads us to a contradiction that $f\in J(z_1)$ and $g\in J(z_2).$ So $(z_1,z_2)$ does not form a parallel pair in $\mathbb{Z}.$ Similarly, if $\|x_1\|<\|y_1\|$ and $\|x_2\|>\|y_2\|$ then    $(z_1,z_2)$ is not a parallel pair in $\mathbb{Z}.$
\end{proof}

	Next, we turn our attention to the preservation of  parallel pairs and TEA pairs by a bounded linear operator in normed linear spaces.  First, we have the following observation that the preservation of TEA pairs always implies the preservation of parallel pairs.
	
	\begin{proposition}
		Let $\mathbb{X}, \mathbb{Y}$ be normed linear spaces. If $T \in \mathbb{L}(\mathbb{X},\mathbb{Y})$ preserves TEA pairs then $T$ preserves parallel pairs.
	\end{proposition}
	\begin{proof}
		Let  $T \in \mathbb{L}(\mathbb{X},\mathbb{Y})$ preserve TEA pairs. Let $x,y\in \mathbb{X}$ form a parallel pair in $\mathbb {X}.$ Then there exists a  scaler  $\lambda$ with $|\lambda|=1$ such that 
		\[\|x+\lambda y\|=\|x\|+\|y\|=\|x\|+\|\lambda y\|.\]
		This implies that $(x,\lambda y)$ is a TEA pair in $\mathbb{X}.$ Since $T$ preserves TEA pairs, it follows that $(Tx,\lambda Ty)$ is a TEA pair in $\mathbb{Y}$ and so 
		\[\|Tx+\lambda Ty\|=\|Tx\|+\|\lambda Ty\|=\|Tx\|+\|Ty\|.\]
		Thus, $(Tx,Ty)$ is a parallel pair in $\mathbb{Y}.$ Therefore, $T$ preserves parallel pairs.
	\end{proof}
	In the following example, we show that the preservation of parallel pairs does not necessarily guarantee the preservation of TEA pairs.
	\begin{example}
		Let $\mathbb{X}=\ell_{1}^3.$ Define $T\in \mathbb{L}(\mathbb{X})$ by $T(x,y,z)=(x-z,0,x-z)$ for $(x,y,z)\in \mathbb{X}.$  Clearly, $T$  preserves parallel pairs. Now,  $u=(2,1,1), v=(1,1,2)$ form a TEA pair but $Tu=(1,0,1), Tv=(-1,0,-1)$ do not form a TEA pair.
	\end{example}

    Now, we show that if a bounded linear operator preserves TEA pairs of any face of the unit ball of a normed linear space, then the image of any relative interior point of that face under the operator cannot be zero unless the image of the entire face is zero. To proceed, we need the following lemma. 
	
	\begin{lemma}\label{pcara}
		Let $\mathbb{X}$ be a normed linear space and let $F$ be a $k$-face  of $B_{\mathbb{X}}$  for some $k\in \mathbb{N}.$ Then for each $x \in \rint F$, there exist $k+1$ linearly independent $x_0,x_1, x_2, \ldots, x_k \in F$ such that
		\[x=\frac{1}{k+1} \sum_{i=0}^k x_i.\]
	\end{lemma}
	\begin{proof}
		Let $x \in \rint F$. Then there exists	$\epsilon>0$ such that $\{y \in\aff(F):\|y-x\| \leq \epsilon\} \subset F.$ Consider $G=F-\{x\},$ i.e., $G=\{y-x: y \in F\}.$ Then $\mathbb{V}=\spn~G$ is a $k$-dimensional subspace of $\mathbb{X}.$ Then it easy to observe  that $\aff(F)-\{x\}=\mathbb{V}.$ This implies that $\{v \in\mathbb{V}:\|v\| \leq \epsilon\} \subset G.$ Let $v_1, v_2, \ldots, v_{k}$ be $k$ linearly independent elements of $\mathbb{V}.$
		Then $z_i=\frac{\epsilon}{k} \frac{v_i}{\|v_i\|} \in G$ for $i=1,2, \dots, k.$ Let $z_0=-\sum_{i=1}^{k} z_i.$
		Then \[\|z_0\| \leq \sum_{i=1}^{k}\|z_i\|=\sum_{i=1}^{k} \frac{\epsilon}{k}=\epsilon.\]
		Thus, $z_0 \in G$ and $\sum\limits_{i=0}^k z_i=0\implies \frac{1}{k+1} \sum\limits_{i=0}^k z_i=0\implies \frac{1}{k+1} \sum\limits_{i=0}^k(z_i+x)=x.$ Since $z_i \in G, \quad z_i+x \in F$ for all $i=1,2, \dots, n.$ Since $z_1,\dots, z_n$ are linearly independent element of $\mathbb{V}$ and $x\notin V,$ it follows  that $z_0+x,z_1+x, \dots, z_{k}+x$ are linearly independent. This completes the proof.
	\end{proof}   
	Now, we have the desired result.
	\begin{theorem}\label{intn0}
		Let $\mathbb{X},\mathbb{Y}$ be  normed linear spaces and let $F$ be a $k$-face of $B_{\mathbb{X}}$ for some $k\in \mathbb{N}.$ If $T\in\mathbb{L}(\mathbb{X},\mathbb{Y})$ preserves  TEA pairs contained in $F$ then  either $T(F)=0$ or $\rint F\cap \ker T=\emptyset.$
	\end{theorem}
	\begin{proof}
		Let  $T\in\mathbb{L}(\mathbb{X},\mathbb{Y})$ preserve TEA pairs of $F.$ Let $T(F)\neq 0.$ If possible suppose that $T(v)= 0$ for some $v\in \rint F.$ From Lemma \ref{pcara}, it follows that there exist $k+1$ linearly independent elements $v_0,v_1,\dots,v_k\in F$ such that $v=\frac{1}{k+1}\sum\limits_0^kv_i.$  Let $m\in \{0,1,2,\dots,k\}$ be arbitrary. Then,
		$ v_m=(k+1)v-\sum\limits_{j=0,\neq m}^kv_j.$
		Now, $ \sum\limits_{j=0,\neq m}^kv_j\in F$ and so 
		$J(v_m)\cap J\Big(\sum\limits_{j=0,\neq m}^kv_j\Big)\neq \emptyset.$
		From Proposition \ref{par char}, it follows that 
		$v_m$ and $\frac{1}{k}\sum\limits_{j=0,\neq m}^kv_j$  form TEA pair and so 	$Tv_m$ and $\sum\limits_{j=0,\neq m}^kTv_j$ form a TEA pair. Now,
		$\sum\limits_{j=0,\neq m}^kTv_j=-Tv_m$ and so $(Tv_m,-Tv_m)$ is a TEA pair. This implies that $Tv_m=0.$ Since $m\in \{0,1,2,\dots,n\}$ is arbitrary, it follows that 
		$Tv_i=0$ for all $i\in \{0,1,2,\dots,n\}.$ This contradicts the fact that $T(F)\neq0.$ Thus, for any face $F$ of $B_{\mathbb{X}},$ $\rint F\cap \ker(T)=\emptyset.$
	\end{proof}
	In the previous result, the conclusion may not hold if the operator is assumed to preserve parallel pairs.
	\begin{example}
		Let $\mathbb{X}=\ell_{1}^3.$ Define $T\in \mathbb{L}(\mathbb{X})$ by $T(x,y,z)=(x-z, 0 ,x-z)$ for $(x,y,z)\in \mathbb{X}.$  Clearly, $T$  preserves parallel pairs. Consider the facet $F=\{(x,y,z)\in S_{\mathbb{X}}: x,y,z>0\}.$ Here $u=(\frac{1}{3},\frac{1}{3},\frac{1}{3})\in \rint F,$ where $F=\co\{(1,0,0),(0,1,0),(0,0,1)\}$ is a facet of $B_{\mathbb{X}}.$ But $Tu=0.$
	\end{example}
   
	\begin{remark}
	From \cite[Th. 2.24]{SSP24}, it follows that for an $n$-dimensional  polyhedral Banach space $\mathbb{X},$ $x\in S_{\mathbb{X}}$ is a $k$-smooth point of $\mathbb{X}$ if and only if it is a relative  interior point of a $(n-k)$-face of $B_\mathbb{X}.$ Moreover, every $p$-face  of $B_\mathbb{X}$ must contain $p$ number of linearly independent elements of $\mathbb{X}.$\\ Let $\mathbb{Y}$ be any Banach space and let $T\in \mathbb{L}(\mathbb{X}, \mathbb{Y})$ preserve TEA pairs. Assume that the rank of $T$ is $p$ for some $1\leq p\leq n.$ Then for any $1\leq k< p,$ the image under $T$ of a entire $(n-k)$-face can not be zero. Consequently, by Theorem \ref{intn0} that $Tu\neq 0$ for every $u$ contained in the relative interior of any $(n-k)$-face for $1\leq k< p.$ It follows that the image under $T$ of any $k$-smooth point can not be zero  for $1\leq k<p.$ Thus, we conclude that if a operator $T$  of rank $p$ on an $n$-dimensional polyhedral Banach space $\mathbb{X}$ that  preserves TEA pairs then  $\ksm(\mathbb{X})\cap \ker T=\emptyset$ for all $k<p.$ 
	\end{remark}

	In the Following theorem, we give a complete characterization of the preservation of parallel pairs in finite-dimensional Banach  spaces.
	\begin{theorem}\label{ParfacetTosmf}
		Let $\mathbb{X},\mathbb{Y}$ be finite-dimensional Banach  spaces. Let $T\in \mathbb{L}(\mathbb{X},\mathbb{Y})$ and let  $F$ be a face of $B_{\mathbb{X}}$ such that $T(F)\neq \{0\}.$  Then  $T$ preserves  parallel  pairs contained in $F$ if and only if one of the following holds:
		\begin{itemize}
			\item[(i)] $\dim(\spn T(F))=1.$
			\item[(ii)] There exists $u \in  F$  such that  $J(Tu) \subset J(Tv)$ for all  $v\in F\setminus \ker(T).$ 
		\end{itemize}
	\end{theorem}
	\begin{proof}
		We only prove the necessary part as sufficient part of the theorem is obvious. Consider $m=\min\{1\leq k \leq \dim(\mathbb{X}) : Tx\in\ksm{\mathbb{Y}}, x\in F\}.$ Let $u\in F$ be such that $Tu\in m\text{-}\sm(\mathbb{Y}).$ Suppose that $\dim(\spn T(F))\geq 2.$ Let $v\in F$ be such that $Tu$ and $Tv$ are linearly independent.  Since $u,v\in F,$ it follows that $(u,v)$ is a parallel  pair in $\mathbb {X}$ and so $(Tu, Tv)$ is a parallel pair in $\mathbb {Y}.$ From Proposition \ref{par char}, it follows that either $J(Tu)\cap J(Tv)\neq \emptyset$ or $J(-Tu)\cap J(Tv)\neq \emptyset.$\\
		\noindent
		\textbf{Case 1:}
		Suppose $J(Tu)\cap J(Tv)\neq \emptyset.$ Let $g\in J(Tu)\cap J(Tv).$ Let $w=\frac12(u+v).$ Then 
		\[\frac12(\|Tu\|+\|Tv\|)\geq\|Tw\|\geq\|g(Tw)\|=g(\frac12(Tu+Tv))=\frac12(\|Tu\|+\|Tv\|).\]
		Therefore, $\|Tu\|+\|Tv\|=2 \|Tw\|.$ Next, we claim that $J(Tw)\subset J(Tu)\cap J(Tv).$ Let $\phi\in J(Tw).$ Now,
		\begin{eqnarray*}
			&&\frac{Tw}{\|Tw\|}=\frac{\|Tu\|}{2\|Tw\|}\frac{Tu}{\|Tu\|}+\frac{\|Tv\|}{2\|Tw\|}\frac{Tv}{\|Tv\|}\\
			&\implies& \phi\left(\frac{Tw}{\|Tw\|}\right)=\frac{\|Tu\|}{2\|Tw\|}\phi\left(\frac{Tu}{\|Tu\|}\right)+\frac{\|Tv\|}{2\|Tw\|}\phi\left(\frac{Tv}{\|Tv\|}\right)\\
			&\implies& 1=\frac{\|Tu\|}{2\|Tw\|}\phi\left(\frac{Tu}{\|Tu\|}\right)+\frac{\|Tv\|}{2\|Tw\|}\phi\left(\frac{Tv}{\|Tv\|}\right).
		\end{eqnarray*}
		Since $\frac{\|Tu\|}{2\|Tw\|}+\frac{\|Tv\|}{2\|Tw\|}=1$,
		it follows that 	
		\[\phi\left(\frac{Tu}{\|Tu\|}\right)=1 \text{ and } \phi\left(\frac{Tv}{\|Tv\|}\right)=1\] 
		Therefore, $\phi\in J(Tu)\cap J(Tv)$ and this establishes our claim that $J(Tw)\subset J(Tu)\cap J(Tv).$  Since $Tu$ is an $m$-smooth point and the order of smoothness of $Tw$ is at least $m,$ it follows from $J(Tw)\subset J(Tu)$ that $J(Tu)=J(Tw).$  This implies that $J(Tw)= J(Tu)\subset J(Tv).$
		
		\textbf{Case 2:} Suppose $J(-Tu)\cap J(Tv)\neq\emptyset.$  Consider
		\begin{eqnarray*}
			t_0&=&\sup\{t\in [0,1]:J(Tu)\cap J((1-t)Tu+tTv)\neq \emptyset \}\\ \text{ and } t_1&=&\inf\{t\in [0,1]:J(-Tu)\cap J((1-t)Tu+tTv)\neq \emptyset \}.
		\end{eqnarray*}
		If there exists $\lambda>0$ such that $J(Tu)\cap J((1-\lambda )Tu+\lambda Tv)\neq \emptyset,$ then it is easy to observe that  $J(Tu)\cap J((1-t )Tu+t Tv)\neq \emptyset$  for all $0\leq t \leq \lambda.$ Next,
		it follows from \textbf{Case 1} that for each $0<\mu\leq  t_0,$   $ J(Tu)\subset J((1-\mu)Tu+\mu Tv).$ Since $J(Tu)\cap J(-Tu)= \emptyset,$ it follows that $t_0<t_1.$ Let $t_2=\frac{t_0+t_1}{2}.$ Let $w=(1-t_2)u+t_2v.$ Then
		\[J(Tu)\cap J(Tw)= \emptyset \text{ and } J(-Tu)\cap J(Tw)= \emptyset.\]
		Hence $(Tu,Tw)$ is not a parallel pair in $\mathbb {Y}.$ This contradicts the fact that $(u,w)$  is a parallel pair in $\mathbb {X}.$ Thus, $J(-Tu)\cap J(Tv)=\emptyset.$ 
		
		Therefore, for all  $u,v\in F$ such that  $Tu,Tv$ are linearly independent with $Tu\in m\text{-}\sm(\mathbb{Y}),$  we have $J(Tu)\subset J(Tv).$ \\
        Now,  let  $u,v\in F \setminus \ker(T)$  such that  $Tu$ and $Tv$ are linearly dependent $m$-smooth point of $\mathbb{Y}.$ Let $z\in F$ be such that $Tu$ and $Tz$ are linearly independent.
		Then $J(Tu)\subset J(Tz)$ and $J(Tv)\subset J(Tz).$ Therefore, $J(Tu)=J(Tv).$ Thus, for each  $u\in F$ such that $Tu$ is  an  $m$-smooth point of $\mathbb{Y},$ we have $J(Tu)\subset J(Tv)$ for all $v\in F\setminus \ker(T).$ 
	\end{proof}
	Next, we obtain a complete characterization of the preservation of TEA pairs in finite-dimensional Banach  spaces.
	\begin{theorem}\label{TEAfacetTosmf}
		Let $\mathbb{X},\mathbb{Y}$ be finite-dimensional Banach  spaces. Let $T\in \mathbb{L}(\mathbb{X},\mathbb{Y})$ and let  $F$ be a face of $B_{\mathbb{X}}$ such that $T(F)\neq \{0\}.$  Then  $T$ preserves  TEA  pairs contained in $F$ if and only if  there exists $u \in  F$  such that  $J(Tu) \subset J(Tv)$ for all  $v\in F\setminus \ker(T).$
	\end{theorem}
	\begin{proof}
		Sufficient part of the theorem is obvious. So we only prove the necessary part.
		Consider $m=\min\{1\leq k\leq \dim(\mathbb{X}): Tx\in\ksm{\mathbb{Y}},~ x\in F\}.$ Let $u\in F$ be such that $Tu$ is an $m$-smooth point of $    \mathbb{Y}.$ We show that for each $v\in F\setminus ker(T),$  $J(u)\subset  J(v).$ Let $v\in F \setminus ker(T).$ Since $u,v\in F,$ it follows that $u,v$ form a TEA pair in $\mathbb {X}$ and so $Tu, Tv$ form a TEA pair in $\mathbb {Y}$  and so $J(Tu)\cap J(Tv)\neq 0.$  If $Tu\text{ and }Tv$ are linearly dependent then $J(Tu)=J(Tv).$ Let $Tu\text{ and }Tv$ are linearly independent. Then the result follows by using the same argument using in the \textbf{Case 1} of the proof of the previous theorem.   
	\end{proof}
   It has already been observed that a bounded linear operator that preserves parallel pairs may not preserve TEA pairs. However, Theorems \ref{ParfacetTosmf} and \ref{TEAfacetTosmf} together imply that, in the case of finite-dimensional Banach spaces, if the operator is bijective, then the preservation of parallel pairs does guarantee the preservation of TEA pairs. 
	\begin{cor}
		Let $\mathbb{X},\mathbb{Y}$ be finite-dimensional Banach spaces. Let $T\in \mathbb{L}(\mathbb{X},\mathbb{Y})$ be a bijective operator. 
		Then $T$ preserves parallel pairs if and only if $T$ preserves TEA pairs.
	\end{cor}
	
    Additionally, in finite-dimensional polyhedral Banach spaces,  this result holds true even when the operator has rank at least 2.

	\begin{cor}\label{corchar}
		Let $\mathbb{X},\mathbb{Y}$ be finite-dimensional polyhedral Banach spaces. Let $T\in \mathbb{L}(\mathbb{X},\mathbb{Y})$ be such that rank$(T)\geq2.$ Then the following are equivalent:
			\begin{itemize}
					\item[(i)] $T$ preserves parallel  pairs.
			\item[(ii)] $T$ preserves TEA  pairs.
			
			\item[(iii)] For each facet $F$ of $B_{\mathbb{X}},$ there exists $u \in  F$  such that  $J(Tu) \subset J(Tv)$ for all  $v\in F\setminus \ker(T).$ 
		\end{itemize}
	\end{cor}
	
	For each $f \in \ext~B_{\mathbb{X}^*},$ we denote the set of all smooth points supported by $f$ by $\sm(f),$ i.e., $\sm(f)=\{x \in \mathbb{X}: J(x)=\{f\}\}.$ We state the following proposition which will help us to establish  relations between the preservation of parallel pairs and preservation of order of smoothness of the points.
	
	\begin{proposition}\cite[Prop. 2.1]{MMPS25}\label{propo}
		Let $\mathbb{X}$ be a finite-dimensional polyhedral Banach space. Then the following results  hold:
		\begin{itemize}
			\item[(i)] For each $f \in \ext~B_{\mathbb{X}^*}$, $\sm(f)$ is nonempty, open and a convex cone.
			\item[(ii)] For each  non-zero $x\in \mathbb{X},$ $f\in \ext~J(x)$ if and only if $x\in \overline{\sm(f)} .$
			\item[(iii)] $\mathbb{X}=\bigcup\limits_{f \in \ext~B_{\mathbb{X}^*}}\overline{\sm(f)}$.
			\item[(iv)] If $f=-g$, then $\overline{\sm(f)} \cap \overline{\sm(g)}=\{0\}$.
			\item[(v)] Let $F$ be the facet of $B_{\mathbb{X}}$ corresponding to $f \in \ext~ B_{\mathbb{X}^*}$, then ${\sm(f)}=\{r x: r>0\text{ and }x \in \text{Int}_r~ F\},$ i.e.,  ${\sm(f)}\cup \{0\}$  is the conical hull of $\text{Int}_r~ F.$ 
			
		\end{itemize}
	\end{proposition}
	
	Next, we demonstrate that for finite-dimensional polyhedral Banach spaces if any non-zero operator preserves parallel (or TEA) pairs, then it necessarily preserves smooth points, i.e., the image of every smooth point in the domain is a smooth point in the codomain. To establish this, we recall the following lemma.
	\begin{lemma}\cite[Th. 2.8]{MMPS25}\label{openmap}
		Let $\mathbb{X}, \mathbb{Y}$ be Banach spaces and $T \in \mathbb{L}(\mathbb{X},\mathbb{Y})$ be bijective. Suppose that $A$ and $B$ are  open sets of $\mathbb{X}$ and $\mathbb{Y},$ respectively. If $T(A) \cap \overline{B} \neq \emptyset$ then $T(A) \cap B \neq \emptyset.$
	\end{lemma}
    We are now ready to prove the desired result.
	\begin{theorem}\label{facetTofacet}
		Let $\mathbb{X},\mathbb{Y}$ be finite-dimensional polyhedral Banach spaces.  Let $T\in \mathbb{L}(\mathbb{X},\mathbb{Y})$ be a bijective operator. Then $T$ preserves parallel (or TEA) pairs if and only if for each $f \in  \ext~ B_{\mathbb{X}^*}$ there exists a unique $g \in  \ext~ B_{\mathbb{Y}^*}$ such that $ T(\sm(f)) \subset \sm(g).$
	\end{theorem}
	\begin{proof}
    First, we prove the sufficient part.  Let $(x,y)$ be a TEA pair in $\mathbb {X}.$ From Proposition \ref{par char} it follows that $J(x)\cap J(y)\neq \emptyset.$ Then it is easy to observe that  $\ext~J(x)\cap \ext~J(y)\neq \emptyset.$ Let $\phi\in \ext~J(x)\cap \ext~J(y).$ Then it follows from  Proposition \ref{propo}  that $x,y \in \overline{\sm(\phi)}.$ Then by the hypothesis there exists $\psi \in  \ext~ B_{\mathbb{Y}^*}$ such that $ Tx,Ty \in T(\overline{\sm(\phi)}) \subset \overline{\sm(\psi)}.$ Again, it follows from Proposition \ref{propo} that  $\psi\in \ext~J(Tx)\cap \ext~J(Ty).$ Then it follows from Proposition \ref{par char} that $(Tx,Ty)$ is a TEA pair in $\mathbb {Y}.$ This implies that $T$ preserves TEA pairs and so $T$ preserves parallel pairs.
    
    Next, we prove the necessary part. Let $f \in  \ext~ B_{\mathbb{X}^*}.$ Consider the facet $F=\{x\in S_{\mathbb{X}}: f(x)=\|x\|\}.$  From Corollary \ref{corchar}, it follows that there exists $u \in  F$ such that $J(Tu) \subset J(Tv)$ for all $v\in F.$ Let $g\in \ext~J(Tu).$ Then  for all $v\in F,$ $g\in J(Tv).$  Then it follows from  Proposition \ref{propo}  that $T(F)\subset \overline{\sm(g)}$ and so $T(\sm(f))\subset \overline{\sm(g)}.$
		Now, we show that $ T(\sm(f)) \subset \sm(g)$. On the contrary, suppose that there exists $x \in \sm(f)$ such that	$T x \in \overline{\sm(g)} \text { but } Tx \notin \sm(g).$ Then it follows from  Proposition \ref{propo}  that there exists $h(\neq g) \in  \ext~ B_{\mathbb{Y}^*}$ such that $h \in J(Tx).$ So $Tx \in \overline{\sm(h)}$. From Lemma \ref{openmap}, it follows that there exist $v_1, v_2 \in \sm(f)$ such that $Tv_1 \in \sm(g)$ and $Tv_2 \in \sm(h)$. Again, this leads to a contradiction. Therefore, $T(\sm(f)) \subset {\sm(g)}$.\\
        Now, for any   $g_1,g_2 \in  \ext~ B_{\mathbb{Y}^*}$ $g_1\neq g_2$ implies that $\sm(g_1)\cap \sm(g_2)=\emptyset.$ Thus for each $f \in  \ext~ B_{\mathbb{X}^*}$ there exists a unique $g \in  \ext~ B_{\mathbb{Y}^*}$ such that $ T(\sm(f)) \subset \sm(g).$ 
	\end{proof}
	If the operator in the above theorem is not bijective then the result may fail to hold.      
	\begin{example}
		Let $\mathbb{X}=\ell_{1}^3.$ Define $T\in \mathbb{L}(\mathbb{X})$ by $T(x,y,z)=(x, 0,z)$ for $(x,y,z)\in \mathbb{X}.$  Clearly, $T$ is not bijective and $T$ preserve parallel pairs. Here $u=(1,1,1)$ is a smooth point but $Tu=(1,0,1)$ is not a smooth point.
	\end{example}
 Next, we show that in finite-dimensional polyhedral Banach spaces, there exists $\epsilon >0$ such that every non-zero operator that  preserves $\epsilon$-approximate Birkhoff-James orthogonality must also preserve TEA pairs. For this purpose,  recall that \cite{C05} for $\epsilon\in[0,1),$  $x\in \mathbb {X}$ is  said to be $\epsilon$-approximate Birkhoff-James orthogonal ($\epsilon$-orthogonal) to $y\in \mathbb{X},$ denoted by $x\perp_B^{\epsilon}y,$ if for every scalar $\lambda,$
		\[\|x+\lambda y\|^2\geq \|x\|^2-2\epsilon\|x\|\|\lambda y\|.\]
We define the $\epsilon$-orthogonal set of  $x\in \mathbb{X}$ as  $x^{\perp_B^\epsilon} = \{ y \in \mathbb{X} : x \perp_B^\epsilon y \}. $
        
An operator $T\in\mathbb{L}(\mathbb{X},\mathbb {Y})$ is said to  preserve  $\epsilon$-orthogonality at $x\in\mathbb{X},$ if for any  $y\in \mathbb{X}, ~ x\perp_B y\implies Tx\perp_B^{\epsilon} Ty.$\\
 We also recall the constant $\epsilon_{\mathbb{X}}$ for an $n$-dimensional polyhedral Banach space
$\mathbb{X},$ as defined in \cite{MMPS25}, given by
	\[\epsilon_{\mathbb{X}}= \min\{\epsilon_{\phi\psi}: \phi,\psi\in  \ext~ B_{\mathbb{X}^*} \text{ are linearly independent}\},\]
     where $\epsilon_{\phi\psi}>0$ such that $x^{\perp_B^{\epsilon_{\phi\psi}}}\cap y^{\perp_B^{\epsilon_{\phi\psi}}}$ does not contain any hyperspace, for any $x\in\sm(\phi)$ and $y\in \sm(\psi).$  
    
		\begin{theorem}\label{aprox to par}
		Let $\mathbb{X},\mathbb{Y}$ be $n$-dimensional polyhedral Banach spaces and let $0\leq \epsilon<\epsilon_{\mathbb{Y}}.$ If a non-zero $T\in \mathbb{L}(\mathbb{X},\mathbb{Y})$ preserves $\epsilon$-orthogonality at each $x \in \mathbb{X}$, then  $T$ preserves TEA pairs.
	\end{theorem}
    \begin{proof}
        Let  $T\in \mathbb{L}(\mathbb{X},\mathbb{Y})$ preserve $\epsilon$-orthogonality at each $x \in \mathbb{X}.$ Then from \cite[Th. 2.10]{MMPS25}, it follows that for each $f \in  \ext~ B_{\mathbb{X}^*}$ there exists a unique $g \in  \ext~ B_{\mathbb{Y}^*}$ such that $ T(\sm(f)) \subset \sm(g)$ and so $ T(\overline{\sm(f)}) \subset \overline{\sm(g)}.$  Then it follows from Theorem \ref{facetTofacet} that $T$ preserves parallel TEA  pairs.
    \end{proof}
		Now, we explore the  connection between the number of facets of the unit balls of finite-dimensional polyhedral Banach spaces and the preservation of parallel (or TEA) pairs by a bijective operator.
	\begin{theorem}\label{distinct}
		Let $\mathbb{X},\mathbb{Y}$ be finite-dimensional polyhedral Banach spaces.    Let a bijective operator $T\in \mathbb{L}(\mathbb{X},\mathbb{Y})$ preserve parallel (or TEA) pairs. Then the following results hold:
		\begin{itemize}
			\item[(i)] $|\ext~ B_{\mathbb{X}^*}|\geq|\ext~ B_{\mathbb{Y}^*}|.$ 
			\item[(ii)]  $|\ext~ B_{\mathbb{X}^*}|=|\ext~ B_{\mathbb{Y}^*}|$ if and only if  for any two distinct $f_1,f_2\in \ext~ B_{\mathbb{X}^*}$ there exist two distinct $g_1,g_2\in \ext~ B_{\mathbb{Y}^*} $ such that  $T(\sm(f_1)) \subset\sm(g_1)$ and $T(\sm(g_1)) \subset \sm(g_2).$
		\end{itemize}
	\end{theorem}
	\begin{proof}

		(i)	If possible suppose that $|\ext~ B_{\mathbb{X}^*}|<|\ext~ B_{\mathbb{Y}^*}|.$ Theorem \ref{facetTofacet} ensures that for every  $f\in \ext~B_{\mathbb{X}^*}$ there exists a unique $g \in \ext~ B_{\mathbb{Y}^*}$ such that $T(\sm(f)) \subset \sm(g).$ Then there exists $g\in \ext~ B_{\mathbb{Y}^*}$ such that there does not exist any $f\in \ext~B_{\mathbb{X}^*}$ such that $T(\sm (f))\subset \sm (g).$
		 Let $v\in \sm (g).$  Since $T$ is bijective, there exists a non-zero $u\in \mathbb{X}$ such that $Tu=v.$ Let $\phi\in \ext~B_{\mathbb{X}^*} $ such that $\phi\in J(u).$ Then it follows from  Proposition \ref{propo}  that $u\in \overline{\sm(\phi)}.$
		Then there exists a unique $\psi \in \ext~ B_{\mathbb{Y}^*}$ such that $T(\sm(\phi)) \subset \sm(\psi).$
		Therefore, $Tu=v\in \overline{\sm(\psi)}\cap \sm(g).$ If $g\neq \psi$ then $\overline{\sm(\psi)}\cap \sm(g)=\emptyset,$ which is not possible. Thus, $\psi=g$ and so  $T(\sm(\phi)) \subset \sm(g),$ which is a contradiction. Thus $|\ext~ B_{\mathbb{X}^*}|\geq|\ext~ B_{\mathbb{Y}^*}|.$

		(ii) Let $|\ext~ B_{\mathbb{X}^*}|=|\ext~ B_{\mathbb{Y}^*}|.$ If possible suppose that there exist two distinct $f_1,f_2\in \ext~ B_{\mathbb{X}^*}$ such that   $T(\sm(f_1)) \cup T(\sm(f_2))  \subset\sm(g)$  for some $g\in  \ext~ B_{\mathbb{Y}^*}. $ Then there exists $\psi\in \ext~ B_{\mathbb{Y}^*}$ such that there does not exist any $\phi\in \ext~B_{\mathbb{X}^*}$ such that $T(\sm (\phi))\subset \sm (\psi).$ Now, using similar arguments as in (i), we arrive at a contradiction. Thus,  for any two distinct $f_1,f_2\in \ext~ B_{\mathbb{X}^*}$ there exist two distinct $g_1,g_2\in \ext~ B_{\mathbb{Y}^*} $ such that  $T(\sm(f_1)) \subset\sm(g_1)$ and $T(\sm(g_1)) \subset \sm(g_2).$\\
		Conversely, let for any two distinct $f_1,f_2\in \ext~ B_{\mathbb{X}^*}$ there exist two distinct $g_1,g_2\in \ext~ B_{\mathbb{Y}^*} $ such that  $T(\sm(f_1)) \subset\sm(g_1)$ and $T(\sm(g_1)) \subset \sm(g_2).$ Then  Theorem \ref{facetTofacet} ensures that  $|\ext~ B_{\mathbb{X}^*}|\leq|\ext~ B_{\mathbb{Y}^*}|.$  From (i), it follows that $|\ext~ B_{\mathbb{X}^*}|\geq|\ext~ B_{\mathbb{Y}^*}|.$ Therefore, $|\ext~ B_{\mathbb{X}^*}|=|\ext~ B_{\mathbb{Y}^*}|.$
	\end{proof}
	For the finite-dimensional polyhedral Banach spaces whose unit spheres have the same number of facets, we have the following result.
	\begin{theorem}\label{cardpreserve1}
		Let $\mathbb{X},\mathbb{Y}$ be $n$-dimensional polyhedral Banach spaces such that $|\ext~ B_{\mathbb{X}^*}|$ $=|\ext~ B_{\mathbb{Y}^*}|.$  If a bijective operator  $T \in \mathbb{L}(\mathbb{X},\mathbb{Y})$ preserves parallel (or TEA) pairs then the following results hold true.
		\begin{itemize} \label{card fac equ}
			\item[(i)]	For each $f \in  \ext~ B_{\mathbb{X}^*}$ there exists a unique  $g \in  \ext~ B_{\mathbb{Y}^*}$ such that  $T(\sm(f)) = \sm(g).$ 
			\item[(ii)] For each non-zero $x\in \mathbb{X},$  $|\ext ~J(x)| =|\ext ~J(T x)|$.
		\end{itemize}
	\end{theorem}
	\begin{proof} 
		(i)  	Let $f\in \ext~B_{\mathbb{X}^*}.$ From Theorem \ref{facetTofacet}, it follows that  $T(\sm(f)) \subset \sm(g), $ for a unique $g \in \ext~ B_{\mathbb{Y}^*}.$ We show that   $T(\sm(f)) = \sm(g). $ On the contrary suppose that $y\in \sm(g)\setminus T(\sm(f)).$ Since $T$ is bijective, there exists $x(\neq 0)\in \mathbb{X}$ such that $Tx=y.$ As $y\notin T(\sm(f)),$ it follows that $x\notin  \sm(f)$ and hence,  it follows from  Proposition \ref{propo}  that there exists $f_1(\neq f)\in \ext~B_{\mathbb{X}^*} $ such that $f_1\in \ext~J(x).$ Thus $x\in \overline{\sm(f_1)}.$  From Theorem \ref{distinct} (ii), it follows that there exists  $g_1(\neq g) \in \ext~ B_{\mathbb{Y}^*}$ such that $T(\sm(f_1)) \subset \sm(g_1).$ Thus, $Tx\in\overline{\sm(g_1)} $ and so $Tx=y\in \overline{\\sm(g_1)}\cap\sm(g)=\emptyset.$ This is a contradiction. Thus, $T(\sm(f)) = \sm(g). $\\

		(ii) 	Let $x(\neq 0) \in \mathbb{X}$ and let $f\in \ext~J(x)$ be arbitrary. From Proposition \ref{propo}, it follows that $x\in \overline{ \sm(f)}.$ Now, it follows from (i) that there exists a unique $g \in \ext~ B_{\mathbb{Y}^*}$ such that $T(\sm(f)) = \sm(g).$ Thus, $Tx\in  \overline{\sm(g)}$ and so $g\in \ext~J(Tx).$  Therefore, for each $f \in \ext~ J(x),$ there exists a unique $g\in \ext~J(Tx)$ such that  $T(\sm(f))= \sm(g).$ From Theorem \ref{distinct} (ii), it follows that $|\ext ~J(x)| \leq |\ext ~J(T x)|$. 
		If possible let $|\ext ~J(x)|< |\ext ~J(T x)|$.
 	 Then there exists $\xi\in \ext~J(Tx)$  such that  there does not exist $h\in \ext ~J(x)$ such that $T(\sm(h))= \sm(\xi).$ Since $|\ext~ B_{\mathbb{X}^*}|$ $=|\ext~ B_{\mathbb{Y}^*}|,$ it follows that  there exists $\zeta \in \ext~ B_{\mathbb{X}^*}$ such that $T(\sm(\zeta)) = \sm(\xi)$ and so $\zeta \notin \ext~J(x).$ Now, $Tx\in \overline{ \sm(\xi)}.$ Let a sequence $\{x_n\}\subset\sm(\zeta)$ be such that $Tx_n\longrightarrow Tx.$ Now, as $T$ is bijective, $T^{-1}$ exists. This implies that $T^{-1}(Tx_n)\longrightarrow  T^{-1}(Tx)\implies x_n \longrightarrow  x$ and so $x\in \sm(\zeta).$ From Proposition \ref{propo}, it follows that $\zeta\in \ext~J(x),$ which is a  
 	 contradiction.  Therefore $| \ext~J(x)|= |\ext~J(T x)|$
		
	\end{proof}
	As a direct consequences of the above theorem, we obtain the following two corollaries.
	\begin{cor}\label{non smooth}
		Let $\mathbb{X}$ be a finite-dimensional polyhedral Banach space.   If a bijective operator $T \in \mathbb{L}(\mathbb{X})$ preserves parallel( or TEA) pairs, then $T$ maps each non-smooth point of $\mathbb{X}$ to a non-smooth point of $\mathbb{X}.$
	\end{cor} 
	
	\begin{cor}\label{corsmoothpreserve}
		Let $\mathbb{X}$ be a finite-dimensional polyhedral Banach space. Suppose that for any  $x_1, x_2 \in \mathbb{X},$  the order of smoothness of $x_1>$  the order of smoothness of $x_2~\implies|\ext~J(x_1)|>|\ext~J(x_2)|.$ If a bijective operator $T\in \mathbb{L}(\mathbb{X})$ preserves parallel pairs, then for each $z\in \mathbb{X},$  the order of smoothness of $Tz$ is equal to the order of smoothness of $z.$
	\end{cor}
	
	The result mentioned in the above corollary does not necessarily hold when the operator is not bijective.
	\begin{example}
		Let $\mathbb{X}=\ell_{1}^3.$ Define $T\in \mathbb{L}(\mathbb{X})$ by $T(x,y,z)=(x, 2x,z)$ for $(x,y,z)\in \mathbb{X}.$  Clearly, $T$ is not bijective and $T$ preserve parallel pairs. Here $u=(1,0,0)$ is a $3$-smooth point but $Tu=(1,2,0)$ is a $2$-smooth point.
	\end{example}
	
Finally, we present the following characterization of isometries on a class of Banach spaces. 
	
    	\begin{theorem}\label{isometry}
		Let $\mathbb{X}$ be a finite-dimensional polyhedral Banach space such that for each $u\in \ext~B_{\mathbb{X}},$ $|\ext~J(u)|>|\ext~J(v)|$ for all non-zero $v\in \mathbb{X}\setminus \ext~B_{\mathbb{X}}.$ Let $T\in \mathbb{L}(\mathbb{X})$ be a norm-one operator. Then the following conditions are equivalent: 
		\begin{itemize}
       \item[(i)]  $T$ is an isometry.
       \item[(ii)] $T$ is bijective and  preserves parallel  pairs and $\| Tu\|=\|Tv\|$ for any $u,v \in \ext~B_\mathbb{X}.$
			\item[(iii)] $T$ preserves $\epsilon$-orthogonality for any $0\leq\epsilon<\epsilon_{\mathbb{X}}$ at each point of $\mathbb{X}$ and   $\| Tu\|=\|Tv\|$ for any $u,v \in \ext~B_\mathbb{X}.$
		\end{itemize} 
	\end{theorem}
	\begin{proof}
    (i)$\iff$(iii) follows from \cite[Th. 2.21]{MMPS25}.
	(i)$\implies$(ii) is trivial.
	We only prove	(ii)$\implies$(i). Let  $T\in \mathbb{L}(\mathbb{X})$
		preserve TEA pairs. From theorem \ref{card fac equ} and from the hypothesis,   it follows  that  $T$ maps each  element of $\ext~B_\mathbb{X}$  to a scalar multiple of some element of $\ext~B_\mathbb{X}.$ Since $\mathbb{X}$ is finite-dimensional, $T$ attains its norm at some element of $\ext~B_\mathbb{X}$ and so $\|Tu\| = \|T\| = 1$ for all $u\in \ext~B_\mathbb{X}.$ Thus, $T$ maps elements of $\ext~B_\mathbb{X}$ to elements of $\ext~B_\mathbb{X}.$ As $T$ is bijective, $T^{-1}$ exists and does the same. Now, $T^{-1}$ also attains its norm at some element of $\ext~B_\mathbb{X}$ and so $\|T^{-1}\| = 1$.
		Hence  $\|Tx\|\leq \|x\|=\|T^{-1}(Tx)\|\leq \|Tx\|$ for all $x\in\mathbb{X}.$  This implies $\|Tx\|=\|x\|$ for all  $x\in\mathbb{X}.$ Therefore, $T$ is an isometry. 
	\end{proof}

    \section*{Acknowledgments}
	Jayanta Manna would like to thank UGC, Govt. of India for the support in the form of Senior Research Fellowship under the mentorship of Professor Kallol Paul. 
	
	\section*{Declarations}
	
	\begin{itemize}
		\item Funding
		
		The research of Dr Kalidas Mandal and Professor Kallol Paul is supported by CRG Project bearing File no. CRG/2023/00716 of DST-SERB, Govt. of India.
		
		\item Conflict of interest
		
		The authors have no relevant financial or non-financial interests to disclose.
		
		\item Data availability 
		
		The manuscript has no associated data.
		
		\item Author contribution
		
		All authors contributed to the study. All authors read and approved the final version of the manuscript.
		
	\end{itemize}

\end{document}